\documentclass{amsart}
\usepackage{amssymb,amsthm,url,color}

\usepackage{color}

\newtheorem{theorem}{Theorem}
\newtheorem{lemma}[theorem]{Lemma}
\newtheorem{corollary}[theorem]{Corollary}
\newtheorem{conjecture}[theorem]{Conjecture}

\theoremstyle{definition}

\newcommand{\beg}{\mathop{{\rm beg}}}
\newcommand{\inv}{\mathop{{\rm inv}}}

\newcommand{\p}{\mathrm \wp}

\newcommand{\V}{\mathrm V}
\def\N#1#2{{\bf N}_{{#1}}({{#2}})}
\def\Zent#1{{\bf Z}({#1})}

\newcommand{\ZZ}{\mathbb Z}

\newcommand{\CT}{\mathrm{CT}}
\newcommand{\Aut}{\mathrm{Aut}}

\newcommand{\Sym}{\mathrm{Sym}}

\newcommand{\Tree}{{\mathcal{T}}}
\newcommand{\cH}{{\mathcal{H}}}
\newcommand{\cG}{{\mathcal{G}}}
\newcommand{\cN}{{\mathcal{N}}}
\newcommand{\cP}{{\mathcal{P}}}
\newcommand{\cM}{{\mathcal{M}}}
\newcommand{\Ga}{{\Gamma}}
\newcommand{\tGa}{\tilde{\Gamma}}
\newcommand{\tG}{{\tilde{G}}}
\newcommand{\tM}{{\tilde{M}}}
\newcommand{\tg}{{\tilde{g}}}

\renewcommand{\H}{{\rm H}}

\begin{document}

\title{Lifting a prescribed group of automorphisms  of graphs}

\author[P. Poto\v{c}nik]{Primo\v{z} Poto\v{c}nik}
\address{Primo\v{z} Poto\v{c}nik,\newline
 Faculty of Mathematics and Physics,
 University of Ljubljana \newline 
Jadranska 19, 1000 Ljubljana, Slovenia}\email{primoz.potocnik@fmf.uni-lj.si}

\author[P. Spiga]{Pablo Spiga}
\address{Pablo Spiga, Dipartimento di Matematica Pura e Applicata,\newline
University of Milano-Bicocca,
Via Cozzi~55, 20126 Milano Italy}  \email{pablo.spiga@unimib.it}


\begin{abstract}
In this paper we are interested in lifting a prescribed group of automorphisms of a finite graph via regular covering projections. Here we describe with an example the problems we address and refer to the introductory section for the correct statements of our results.

Let $P$ be the Petersen graph, say, and let $\wp:\tilde{P}\to P$ be a regular covering projection. With the current covering machinery, it is straightforward to find $\wp$ with the property that every subgroup of $\Aut(P)$ lifts via $\wp$. However, for constructing peculiar examples and in applications, this is usually not enough. Sometimes it is important, given a subgroup $G$ of $\Aut(P)$, to find $\wp$ along which $G$ lifts but no further automorphism of $P$ does. For instance, in this concrete example, it is interesting to find a covering of the Petersen graph lifting the alternating group $A_5$ but not the whole symmetric group $S_5$. (Recall that $\Aut(P)\cong S_5$.) Some other time it is important, given a subgroup $G$ of $\Aut(P)$, to find $\wp$ with the property that $\Aut(\tilde{P})$ is the lift of $G$. Typically, it is desirable to find $\wp$ satisfying both conditions. In a very broad sense, this might remind wallpaper patterns on surfaces: the group of symmetries of the dodecahedron is $S_5$, and there is a nice colouring of the dodecahedron (found also by Escher) whose group of symmetries is just $A_5$.

In this paper, we address this problem in full generality.
\end{abstract}

\maketitle

\section{Introduction}

Covering projections of graphs and lifting automorphisms along them is a classical tool in algebraic graph theory that goes back to
Djokovi\'c and his proof of the infinitude of cubic $5$-arc-transitive graphs~\cite{Dj1}. Moreover, several theoretical aspects of
lifting graph automorphisms along covering projections, together with their remarkable applications, are considered in a number of papers, for example in \cite{CM,MNS,elabcov,Sir}, to name a few of the most notable ones.

One of many applications of lifting automorphisms along
covering projections  is the construction of graphs with a prescribed type of symmetry. For example,
covering techniques are used  to find new peculiar examples of semisymmetric
graphs (see \cite{wang,MK,WanChe,ZF}),
 half-arc-transitive 
 graphs (see \cite{CPS,CZ,RS}), and arc-regular graphs (see \cite{CF,feng2,GaSp}), to name a few.

In these papers, a typical strategy is to start with a graph $\Gamma$ and a group $G\le \Aut(\Gamma)$ with a prescribed type of action on 
$\Gamma$
(such as  edge-transitive, vertex-transitive, $s$-arc-transitive for $s\ge 1$, locally primitive etc.) and then trying to find
regular covering projections $\wp:\tGa \to \Gamma$ along which $G$ lifts.
In many
applications, it is desirable for the lift to have the following two additional properties:
\begin{itemize}
\item[(1)] $G$ is the \textit{largest} subgroup of $\Aut(\Gamma)$ that lifts along $\wp$;
\item[(2)] \textit{every} automorphism of $\tGa$ projects along $\wp$.
\end{itemize}
If both these requirements are fulfilled, then $\Aut(\tGa)$ is precisely the lift of $G$.

The problem of finding regular covering projections satisfying (1) has been addressed in an ad-hoc way for some fixed
pairs ($\Gamma,G)$ (see, for example,~\cite{MK,WanChe,ZF}) and determining conditions under which
a covering projection satisfies (2) was considered by several authors in the very specific context of canonical double covers (see~\cite{Sur,Wil}). There have been some attempts to determine covering projections satisfying simultaneously (1) and (2), but again only for a small number of very specific pairs $(\Gamma,G)$ (see, for example,~\cite{CM,Ma,spiga2}).

The aim of this paper is 
to address the problem of existence of a regular covering projection satisfying (1) and (2)
 for arbitrary pairs $(\Gamma,G)$.

In Theorem~\ref{the:main} we prove that, if $\Aut(\Gamma)$ acts faithfully on the integral cycle space $\H_1(\Gamma;\mathbb{Z})$, then
a regular covering projection onto $\Gamma$ fulfilling~(1) always exists; see Section~\ref{BM} for notation and terminology. The condition of $\Aut(\Gamma)$ acting faithfully on $\H_1(\Gamma;\mathbb{Z})$ is very mild: in most interesting cases $\Aut(\Gamma)$ does act faithfully on $\H_1(\Gamma;\mathbb{Z})$, see Lemma~\ref{lemma11} and Corollary~\ref{cor:3ec}. Moreover, there are  examples where $\Aut(\Gamma)$ does not act faithfully on $\H_1(\Gamma;\mathbb{Z})$  and where a regular covering projection as in~(1) does not exist: the easiest example is when $\Gamma$ is a cycle and $G$ is the transitive cyclic subgroup of $\Aut(\Gamma)$.

In Theorem~\ref{the:cor} we prove that, if $\Aut(\Gamma)$ acts faithfully on $\H_1(\Gamma;\mathbb{Z})$ and $(\Gamma,G)$ 
satisfies an additional condition, then there exists a regular covering projection onto $\Gamma$ satisfying~(1) and~(2). The extra condition on $(\Gamma,G)$ is slightly technical and requires some notation and terminology, thus we refer the reader to Section~\ref{sec:Cor} for its definition and details. Here we simply observe that the condition is satisfied by many interesting  classes: for instance,
 when $G$ acts transitively on the $2$-arcs of $\Gamma$, or
when $G$ acts transitively on the arcs of $\Gamma$ and the valency of $\Gamma$  is  prime. 

In Conjecture~\ref{conj}, we dare to conjecture that the additional requirements we put on $\Gamma$ and $G$ in Theorem~\ref{the:cor}
are not needed, that is, 
a regular covering projection satisfying both (1) and (2) exists whenever 
 $\Aut(\Gamma)$ acts faithfully on the integral cycle space of $\Gamma$.

Both Theorems~\ref{the:main} and~\ref{the:cor} do apply to graphs that are not necessarily simple and we refer to Subsection~\ref{subsection}  for the precise definition of graph in our paper.

We conclude this introductory section giving some applications. Theorem~\ref{the:cor} and Corollary~\ref{cor:final} reprove, in a unified way, a number of results that have been proved in the past using methods specific to the families of graphs under consideration. For example, one of the consequences of Corollary~\ref{cor:final} is a solution to three problems posed by Djokovi\'c and Miller~\cite[Problems 2, 3 and 4]{DjM}
about the existence of finite cubic arc-transitive graphs with a prescribed type of the full automorphism group,
 which were first solved in~\cite{ConLor} by an ad-hoc construction. Furthermore, assuming the correctness of Conjecture~\ref{conj} one can answer a 2001 question of Maru\v{s}i\v{c} and Nedela~\cite[Problem 7.7]{MarNed} about the existence of tetravalent half-arc-transitive graphs of any given possible type, and a number
of similar other problems, such as the one of existence of graphs of every possible arc-type, see~\cite{Nem,CPZ}.

\section{Background material and notation}\label{BM}

When working with covering projections of graphs, it proves useful to allow graphs so have multiple edges, loops and
 semiedges: to avoid unnecessary complications, semiedges will be prohibited in this paper. 
 We will thus introduce the definitions and notations pertaining to graphs as defined in~\cite{MNS}, see also~\cite{elabcov} for a more succinct overview. In what follows, we provide only a brief
account.

\subsection{Graph, fundamental group and integral cycle space}\label{subsection}A {\em graph} is an ordered $4$-tuple $\Gamma=(D,V; \beg,\inv)$ where
$D$ and $V$ are disjoint non-empty finite sets of {\em darts}
and {\em vertices}, respectively, $\beg: D \to V$ is a mapping
which assigns to each dart $x$ its {\em tail}
$\beg(x)$, and $\inv: D \to D$ is an involutory permutation which interchanges
every dart $x$ with its {\em inverse dart}, denoted by $x^{-1}$.
The vertex $\beg(x^{-1})$ is then called the {\em head} of the dart $x$.
To avoid degeneracies, we assume that $\inv$ has no fixed points, that is $x^{-1} \not = x$ for every dart $x$.
In the language of \cite{MNS} this means that we only consider graphs without {\em semiedges}. To some extent this hypothesis is not really needed, however it makes the statements of our main results neater and the proofs uniform without subdivision into cases.

An {\em edge} underlying a dart $x$ is an unordered pair $\{x,x^{-1}\}$ of mutually inverse darts,
and $\{\beg(x), \beg(x^{-1})\}$ is the set of {\em endvertices} of that edges. Two edges with the same set of
endvertices are called {\em parallel} and an edge with only one endvertex is a {\em loop}. A graph
without loops and parallel edges is {\em simple} and the usual terminology about simple graphs applies in this case.

The neighbourhood of a vertex $v$ of $\Gamma$, denoted $\Gamma(v)$,
is the set of darts having $v$ as its tail, and the cardinality of $\Gamma(v)$ is called the 
{\em valency} of $v$.

A {\em walk} from a vertex $v$ to a vertex $u$ in $\Gamma$  is a sequence of darts such that 
$v$ is the tail of the first dart, $u$ is the head of the last dart and
the head of each dart in the walk coincides with the tail of the next dart in the walk.  When $v=u$, the walk is said to be {\em closed}, whereas 
the walk is {\em reduced} provided no two consecutive darts are inverse of each other.
Note that if $x$ is a dart and $\{x,x^{-1}\}$ is 
 a loop, then $(x)$ is a reduced closed walk, 
called a 
 {\em loop walk}.
 The empty sequence of darts is considered a walk and is called a trivial walk.

For a vertex $b$ of $\Gamma$ one can define the
{\em fundamental group at $b$}, denoted  $\pi(\Gamma,b)$, as the set of all closed reduced walks starting and
ending in $b$, with the operation being the concatenation (with the deletion of consecutive pairs of mutually inverse darts, if necessary). Note that $\pi(\Gamma,b)$ is a free product of infinite cyclic groups and cyclic groups of order $2$, the latter arising from semiedge walks. In particular, since we are assuming that $\Gamma$ has no semiedges, $\pi(\Gamma,b)$ is a free group.

The abelianisation $\pi(\Gamma,b)/[\pi(\Gamma,b),\pi(\Gamma,b)]$ of $\pi(\Gamma,b)$, viewed
as a $\ZZ$-module, is called the {\em first homology group} or the {\em integral cycle space}
and denoted $\H_1(\Gamma;\ZZ)$. 

For the rest of this subsection, we assume $\Gamma$ to be connected. Then, $\H_1(\Gamma;\ZZ)$ 
is independent on the choice of the vertex $b$. Moreover,
$\H_1(\Gamma;\ZZ)$ is isomorphic to $\ZZ^{m_\Gamma}$ where $m_\Gamma$ is the {\em Betti number of $\Gamma$}, that is, the number of cotree edges relative to a fixed spanning tree of $\Gamma$.
In fact, given a fixed spanning tree $\Tree$ of $\Gamma$, a $\mathbb{Z}$-basis for $\H_1(\Gamma,\ZZ)$ can be chosen
in such a way that each cotree edge $e$ corresponds to an oriented cycle in $\Gamma$ whose only cotree edge is $e$. 
This generating set is called an {\em oriented cycle basis} and does depend on the choice of $\Tree$.

Given a prime number $p$, we let $\ZZ_p$ be the finite field of order $p$. Now, the tensor product $\H_1(\Gamma;\ZZ)\otimes_{\ZZ}\mathbb{Z}_p$ will be denoted $\H_1(\Gamma;\ZZ_p)$. 
Since $\ZZ_p$ is a $(\ZZ,\ZZ_p)$ bi-module, $\H_1(\Gamma;\ZZ_p)$ can be viewed as a $\ZZ_p$-module. 
Note that 
$\H_1(\Gamma;\ZZ_p)\cong \ZZ_p^{m_\Gamma}$.
Indeed, an oriented cycle basis for $\H_1(\Gamma;\ZZ)$ gives rise to a $\mathbb{Z}_p$-basis for $\H_1(\Gamma,\ZZ_p)$ whose elements
correspond to oriented cycles of $\Gamma$.

\subsection{Graph morphism, regular covering projection, universal covering}
Let $\tilde{\Gamma}= (\tilde{D},\tilde{V}; \beg_{\tilde{\Gamma}},\inv_{\tilde{\Gamma}})$ and $\Gamma= (D,V; \beg_\Gamma,\inv_\Gamma)$ be two graphs.
A {\em morphism of graphs}, $f \colon \tilde{\Gamma} \to \Gamma$,
is a function $f \colon \tilde{V} \cup \tilde{D} \to V \cup D$
such that
$f(\tilde{V}) \subseteq V$, $f(\tilde{D}) \subseteq D$,
$f\circ \beg_{\tilde{\Gamma}} = \beg_\Gamma \circ f$ and $f \circ \inv_{\tilde{\Gamma}} = \inv_\Gamma \circ f$.
A graph morphism is an {\em epimorphism} ({\em automorphism}) if it is
a surjection (bijection, respectively).
A graph epimorphism $\wp\colon \tGa \to \Gamma$ is called a {\em covering projection} provided that it maps
the neighbourhood $\tGa(\tilde{v})$ bijectively onto the neighbourhood $\Gamma(\wp(\tilde{v}))$, for every $\tilde{v}\in \tilde{V}$.

Let $\wp\colon \tGa \to \Gamma$ be a covering projection of connected graphs, let
$g\in \Aut(\Gamma)$ and let $\tg\in\Aut(\tGa)$ be such that $\wp(x^\tg) = \wp(x)^g$ for every vertex and for every dart $x$ of $\tGa$.
Then we say that $g$ lifts along $\wp$ and that $\tg$ is a {\em lift} of $g$. Similarly, we say that $\tg$ projects along $\wp$ and
that $g$ is a {\em projection} of $\tg$ along $\wp$. The set of all
automorphisms of $\Gamma$ that lift along $\wp$ is called the {\em maximal group that lifts along $\wp$}. If $G$ is a subgroup
of the maximal group that lifts, then the set $\tG$ of all lifts of elements of $G$ forms a subgroup of $\Aut(\tGa)$ and is called {\em the lift of $G$}. The lift of the maximal group that lifts along $\wp$ is the {\em maximal group that projects along $\wp$}.

The lift of the identity group  is called the {\em group of covering transformations of $\wp$} and denoted $\CT(\wp)$. Whenever the covering graph $\tGa$ is connected, the group $\CT(\wp)$ acts semi-regularly on each fibre, and if it is transitive (and thus regular) on each fibre, then we say that the covering projection
$\wp$ is {\em regular}.

Regular covering projections can equivalently be defined in terms of {\em graph quotients}. Let $\tGa$ be a graph,
let $N \le \Aut(\tGa)$ with the stabiliser 
$N_x$ being trivial for every vertex and for every edge $x$ of $\tGa$,
and let $\tGa/N$ be the graph whose vertices and darts are $N$-orbits of vertices and darts of $\tGa$
and with the functions $\inv_{\tGa/N}$ and $\beg_{\tGa/N}$ mapping a dart $x^N$ of $\tGa/N$ to the $N$-orbit
of $\inv_{\tGa}(x)$ and $\beg_{\tGa}(x)$, respectively (see \cite[Section~2.1]{elabcov}). The corresponding 
{\em quotient projection} $\wp_{N} \colon \tGa \to \tGa/N$, mapping each vertex or dart of $\tGa$ to its $N$-orbit,
is a regular covering projection and $N$ is the group $\CT(\wp_N)$ of covering transformations of $\wp_N$. Every regular covering projection arises in this way.

\begin{lemma}{{\cite[Sections 2.2 and 3]{elabcov}}}
\label{lem:maxlift}
If $\wp \colon \tGa \to \Ga$ is a regular covering projection, then the maximal group that projects along $\wp$
equals the normaliser of $N=\CT(\wp)$ in $\Aut(\tGa)$.
Moreover, $\Ga$ is isomorphic to the quotient graph $\tGa/N$,
the quotient projection $\tGa \to \tGa/N$ is a covering projection isomorphic to $\wp$,
and for a group $G\le \Aut(\Ga)$ and its lift $\tG$, we have $G\cong \tG/N$.
\end{lemma}

\subsection{Splitting of covering projections}

Let $\wp\colon \tGa \to \Ga$ be a regular covering projection between connected graphs
with covering transformation group $N$,  let $G$ be a subgroup of $\Aut(\Ga)$ that lifts along $\wp$ to $\tG$ and let $K$ be a normal subgroup of $N$. Then one can consider the quotient projection $\wp_K \colon \tGa \to \tGa/K$
and define $\wp_{N/K}\colon \tGa/K \to \Ga$ by $\wp_{N/K}(x^K) = \wp(x)$ for every vertex and for every dart $x$ of $\tGa$.
Then the following lemma holds:

\begin{lemma}
The covering projection $\wp_{N/K}$ is regular with covering transformation group
$N/K$ and  $\wp = \wp_{N/K} \circ \wp_K$. Moreover, if $K$ is normalised by $G$, then $G$ lifts along $\wp_{N/K}$ and its lift is $\tG/K$.
\end{lemma}

If $\Tree$ is an infinite tree and $\p\colon \Tree \to \Gamma$ is a regular covering projection, then we say that $\p$
is {\em universal}. It is well known that for every finite graph there is, up to equivalence of covering projections, a unique
universal covering projection and that it has the following property:

\begin{lemma}
\label{lem:unilift}
If $\Gamma$ is a finite connected graph and $\p\colon \Tree \to \Gamma$ is the universal covering projection,
then $\Aut(\Gamma)$ lifts along $\p$ and $\CT(\p)$ is isomorphic to the fundamental group $\pi(\Gamma,b)$
for some (every) vertex $b$ of $\Gamma$. Moreover, $\CT(\p)_x = 1$ for every vertex and for every edge $x$ of $\Gamma$.
\end{lemma}

\section{Results and proofs}

 Given a group $X$ and a subgroup $Y$, we denote by $[X,Y]$ the commutator subgroup defined by $[X,Y]:=\langle x^{-1}y^{-1}xy\mid x\in X,y\in Y\rangle$, by $\N XY$ the normaliser of $Y$ in $X$ and we write $X^p:=\langle x^p\mid x \in X\rangle$. By $\Zent{G}$, we denote the centre of the group $G$.

Recall that a $\mathbb{Z}_pX$-module $V$ is regular if $V$ is isomorphic to the group algebra $\mathbb{Z}_pX$ (seen as a $\mathbb{Z}_pX$-module). This means that $\dim_{\mathbb{Z}_p}(V)=|X|$ and that $V$ has a $\mathbb{Z}_p$-basis $(v_x\mid x\in X)$ such that the action of $X$ on $(v_x\mid x\in X)$ is permutation isomorphic to the action of $X$ on itself by right multiplication. In other words, $v_x y=v_{xy}$, for each $x,y\in X$.

\begin{theorem}
\label{the:tree}
Let $p$ be a prime, let $\Tree$ be an infinite tree, let $\cG \le \Aut(\Tree)$, let $\cN$ be a non-identity normal subgroup of $\cG$ of finite index such that 
$\cN_x = 1$ for every vertex and for every edge $x$ of $\Tree$,
and let $\cH = \N {\Aut(\Tree)} {\cN}$.
If
 $\cH/\cN$ acts faithfully on $\H_{1}(\Tree/\cN;\ZZ)$,
then there exists a normal subgroup $\cP$ of $\cN$ of finite index such that  $\N {\cH}{\cP}=\cG$
and that $\cN/\cP$ is $p$-group.
\end{theorem}

\begin{proof}
The idea for the proof of this theorem is inspired by a surprisingly unrelated problem solved by Bryant and Kov\'{a}cs in~\cite{BK}. We follow closely~\cite{BK} and we use some of the observations therein. This is the second time that this paper on Lie algebras has proved useful in the context of group actions on graph, see for instance~\cite{PSV} for another application.

Observe that, as $\cN\ne 1$ and $\cN_v=1$ for every $v\in \V(\Tree)$, from the Bass-Serre theory, we deduce that $\cN$ is a non-identity free group, see~\cite[Proposition~$4.5$]{Wood}. Following~\cite{BK}, we construct a filtration of the free group $\cN$. Define $\cN_1:=\cN$ and, for $i\in\mathbb{N}\setminus\{0\}$, $\cN_{i+1}:=\cN_i^p[\cN_i,\cN]$.  By construction, $\cN_{i+1}$ is the smallest normal subgroup of 
$\cN$
contained in $\cN_i$ such that $\cN_i/\cN_{i+1}$ has exponent $p$ and is central in $\cN/\cN_{i+1}$, that is, 
$\cN_i/\cN_{i+1}\le \Zent {\cN/\cN_{i+1}}$. Moreover, $\cN_{i+1}$ is normal in $\cG$ because so is $\cN$. 

Write $G:=\cG/\cN_1$ and $H:=\cH/\cN_1$. As $\cG\le \cH$, we have $G\le H$.  Given $i\in\mathbb{N}\setminus\{0\}$, define $V_i:=\cN_i/\cN_{i+1}$.  As $\cN_{i}$ is centralised by $\cN=\cN_1$ modulo $\cN_{i+1}$,  the action of $\cH$ by conjugation on $\cN_i/\cN_{i+1}=V_i$ defines a group homomorphism 
$H=\cH/\cN_1\to \Aut(V_i)$, that is, $H$ acts as a linear group on the $\mathbb{Z}_p$-vector space $V_i$ and hence $V_i$ is a $\mathbb{Z}_pH$-module. The inclusion $G\le H$ allows us to regard, via the restriction mapping, $V_i$ also as $\mathbb{Z}_pG$-modules.

We now require a few facts, from~\cite{Wood} and from~\cite{BK}. From~\cite[Theorem~$9.2$]{Wood}, the $\mathbb{Z}H$-module $\H_1(\Tree/\cN;\mathbb{Z})$ is isomorphic to the $\mathbb{Z}H$-module $\cN_1/[\cN_1,\cN_1]$. Therefore the $\mathbb{Z}_pH$-module $\H_1(\Tree/\cN;\mathbb{Z}_p)=\H_1(\Tree/\cN;\mathbb{Z})\otimes \mathbb{Z}_p$ is isomorphic to the $\mathbb{Z}_pH$-module $\cN/[\cN,\cN]\otimes \mathbb{Z}_p\cong \cN/[\cN,\cN]\cN^p=\cN_1/\cN_2=V_1$. Now, the hypothesis in the statement of the theorem allows us to conclude that $H$ acts faithfully on $V_1=\cN_1/\cN_2$ and hence we can view $H$ as a subgroup of $\Aut(\cN_1/\cN_2)=\Aut(V_1)$. This fact will allow us to apply directly the results from~\cite{BK}. Let $\Sigma=\Aut(V_1)$. From ~\cite[page 416]{BK} it follows that the action of $\Sigma$ on $V_1$ induces an action on $V_i$ and, moreover, the embedding of $H$ in $\Sigma$ is compatible with the action of $H$ defined on $V_i$ above. 

From~\cite[Theorems~2 and 3]{BK}, we deduce that there exists a positive integer $i$ such that
the $\mathbb{Z}_p\Sigma$-module $V_i$ contains a regular submodule. 
Let $R$ be a regular $\mathbb{Z}_p\Sigma$-module contained in $V_{i}$
and let $(r_\sigma \mid \sigma \in \Sigma)$ be a $\ZZ_p$-basis of $R$
with $r_\sigma \delta = r_{\sigma\delta}$ for every $\sigma, \delta \in \Sigma$.
Let $P:=\langle r_\sigma \mid \sigma \in G \rangle$ and observe that $P$ is a regular $\ZZ_pG$-module.

Since $V_i=\cN_i/\cN_{i+1}$, we may write $P=\cP/\cN_{i+1}$ for some subgroup $\cP$  of $\cN_i$ containing $\cN_{i+1}$. Observe that $\cN/\cP$ is a $p$-group because $\cN/\cP$ is a quotient of the $p$-group $\cN_1/\cN_{i+1}$. Moreover, the index of $\cP$ 
in $\cN$ is finite because $\cN_{i+1}$ has finite index in $\cN_1=\cN$. (Each $\mathbb{Z}_p$-vector space $V_i$ is finite dimensional because $\cN_1$ is finitely generated.)

Let $x\in \N \cH \cP=\N{\Aut(\Tree)} \cN\cap \N {\Aut(\Tree)}\cP$. Since $x$ normalises $\cN$, $x$ acts by conjugation as a linear transformation of the vector spaces $\cN_1/\cN_2=V_1$ and $\cN_i/\cN_{i+1}=V_i$.  Denote by $\tau\in \Aut(V_1) = \Sigma$ the linear transformation of $V_1$ induced by the conjugation of $x$. Now, $\tau$ fixes setwise $R$ because $R$ is a $\mathbb{Z}_p\Sigma$-submodule of $V_i$. Since $x$ normalises $\cP$, $\tau$ fixes setwise $\cP/\cN_{i+1}=P$. Since $r_{1} \in P$, we see that  $r_{1} \tau=r_\tau\in P=\langle r_\sigma\mid \sigma\in G \rangle$ and hence $\tau\in G$. Let $y\in \cG$ be an element projecting to $\tau$. Now, $xy^{-1}$ projects to the identity element of $\Sigma=\Aut(V_1)$. Therefore $xy^{-1}$ centralises $V_1=\cN_1/\cN_2$. Observe that $xy^{-1}$ lies in $\cH$ because so does $x$ and $y$. By hypothesis, $\cH/\cN$ acts faithfully on $\H_1(\Tree/\cN;\mathbb{Z}_p)=V_1$. Therefore $xy^{-1}\in \cN$. Since $\cN\le \cG$ and $y\in \cG$, we obtain $x\in \cG$. We have thus shown $\N \cH \cP\le \cG$; the inclusion $\cG\le \N \cH \cP$ is obvious.
\end{proof}

\begin{theorem}
\label{the:main}
Let $p$ be a prime, let $\Gamma$ be a finite connected graph such that the induced action of $\Aut(\Gamma)$
 on $\H_1(\Ga;\ZZ)$ is faithful, and  let $G\le \Aut(\Gamma)$.  
 Then there exists a regular covering projection $\wp \colon \tGa \to \Ga$
with $\tGa$ finite, such that the maximal group that lifts along $\wp$ is $G$ and that the group of covering transformations of $\wp$ is a $p$-group.
\end{theorem}

\begin{proof}
Let $\mu \colon \Tree \to \Gamma$ be the universal covering projection, see Section~\ref{BM}.
 Then $\Tree$ is an infinite tree and 
 in view of Lemma~\ref{lem:unilift},
$G$ lifts along $\mu$ to a group $\cG\le \Aut(\Tree)$. Let $\cN = \CT(\mu)$.
Then $\cN\ne 1$,
 $\cN_x = 1$ for every vertex and for every edge $x$ of $\Tree$, $\Tree/\cN \cong \Gamma$,
and we may identify $\Gamma$ with $\Tree/\cN$ in such a way that $\cG/\cN = G$ and that
 the quotient projection $\wp_\cN \colon \Tree \to \Tree/\cN$ is $\mu$.
 
Let $\cH=\N{\Aut(\Tree)} {\cN}$. By Lemma~\ref{lem:maxlift}, $\cH$ is the largest group that projects along $\mu$ and thus,
since $\Aut(\Gamma)$ lifts along $\mu$, $\cH$ is the lift of $\Aut(\Gamma)$.
  
Since $\Aut(\Ga)$ acts faithfully on $\H_1(\Ga;\ZZ)$, by Theorem~\ref{the:tree}, there exists
a normal subgroup $\cP$ of $\cN$ of finite index such that  $\N{\cH}{\cP}=\cG$ and $\cN/\cP$ is a $p$-group. 

 Let $\tGa = \Tree/\cP$ and $\tG = \cG/\cP \le \Aut(\tGa)$.
In view of Lemma~\ref{lem:maxlift},
the quotient projection $\wp_\cP \colon \Tree \to \tGa$ is a covering projection 
and there exists a regular covering projection $\wp \colon \tGa \to \Ga$ such
that $\mu =  \wp  \circ \wp_\cP$. Moreover, since $\cG$ normalises $\cP$, the group
$G$ lifts along  $\wp$ and its lift is $\tG$.
 
 Let $M\le \Aut(\Gamma)$ be the maximal group that lifts along $\wp$ and let $\tM \le \Aut(\tGa)$ be its lift.
Clearly, $G\le M$ and thus $\tG \le \tM$.
Since $\Tree$ is a tree, $\wp_\cP$ is a universal covering projection, and
in view of Lemma~\ref{lem:unilift},
$\tM$ lifts along $\wp_\cP$ to some $\cM \le \Aut(\Tree)$. 
But then
$\cM$ is the lift of $M$ along $\mu =  \wp  \circ \wp_\cP$,
and thus
$\cM\le \N{\Aut(\Tree)} {\cN}=\cH$. On the other hand, $\cM$ normalises $\cP$ and so 
$\cM\le \N{\cH}{\cP}= \cG$. But then $\tM = \cM/\cP \le \cG/\cP = \tG$, and hence $\tM = \tG$.
Therefore,
$M = \tM/ (\cN/\cP) = \tG/ (\cN/\cP) = G$, and 
 thus $G$ is the maximal group that lifts along $\wp$, as required.
\end{proof}

Let us now discuss the condition of $G$ acting faithfully on $\H_1(\Ga;\ZZ)$.

\begin{lemma}\label{lemma11}
If $\Gamma$ is a simple $3$-edge-connected graph, then $\Aut(\Gamma)$ acts
faithfully on $\H_1(\Ga;\ZZ)$.
\end{lemma}

\begin{proof}
Suppose on the contrary that the action of $\Aut(\Gamma)$ on $\H_1(\Ga;\ZZ)$ is not faithful. Then there exists
an automorphism $g$ fixing every element of $\H_1(\Ga;\ZZ)$ and a vertex $v$ such that $v^g \not = v$.
Let $u,w$ and $z$ be three neighbours of $v$. By $3$-edge-connectivity of $\Gamma$ it follows that
there is a cycle $C_1$ through the $2$-path $uvw$ and a cycle $C_2$ through the $2$-path $uvz$.
Now fix an orientation of $C_1$ and $C_2$ in such a way that $u$ is a predecessor of $v$ in both $C_1$ and $C_2$.
Consider $C_1$ and $C_2$ as elements of $\H_1(\Ga;\ZZ)$. By assumption, $g$ preserves $C_1$ and $C_2$ together with their orientation. In particular, the vertex $v^g$ lies on $C_1$ and on $C_2$. For $i\in \{1,2\}$, let $P_i$ be the path from $v^g$ to $u$ following
the cycle $C_i$ in the positive direction with respect to the chosen orientation. Note that, since $g$ preserves the orientation,
$u^g$ belongs to neither  $P_1$ nor $P_2$. Now consider the closed walk $C$ obtained by concatenating $P_1$ with the reverse of $P_2$ and consider it as an element of $\H_1(\Ga;\ZZ)$. By assumption, $C$ is fixed by $g$. However, the vertex $u$ belongs to $C$, while $u^g$ does not, a contradiction.
\end{proof}

Since in connected vertex-transitive graphs the edge connectivity equals the valency (see for instance~\cite[Lemma~$3.3.3$]{GoRo}), we get the following corollary.
\begin{corollary}
\label{cor:3ec}
If $\Gamma$ a simple connected vertex-transitive graph of valency at least $3$, then $\Aut(\Gamma)$ acts
faithfully on $\H_1(\Ga;\ZZ)$.
\end{corollary}

\section{A corollary}\label{sec:Cor}

Given a graph $\Gamma$, $G\le \Aut(\Gamma)$ and a vertex $v$ of $\Gamma$, we let 
 $G_v^{\Gamma(v)}$  denote the  permutation group
  induced by the action of the vertex-stabiliser $G_v$  on the neighbourhood $\Gamma(v)$ of the vertex $v$.
A finite transitive permutation group $L$ is {\em graph-restrictive} 
provided there exists a constant $c = c(L)$ such that whenever
$\Gamma$ is a connected $G$-arc-transitive group with $G_v^{\Gamma(v)}$ being permutationally isomorphic to $L$,
the order of the stabiliser $G_v$ is at most $c(L)$. This notion was introduced and studied in~\cite{PSV1} and is relevant in the context of the Weiss conjecture: using this terminology, Weiss conjecture states that every primitive group is graph-restrictive.  

We will call a transitive permutation group $L$  acting on a set $\Omega$
{\em strongly graph-restrictive} if every group $T$ with $L\le T \le\Sym(\Omega)$ is graph-restrictive.
Examples of strongly graph-restrictive permutation groups are provided by certain classes of primitive groups. As the culmination of work of Weiss and Trofimov, every $2$-transitive group is graph-restrictive and, as an overgroup of a $2$-transitive group is still $2$-transitive, we deduce that $2$-transitive groups are strongly graph-restrictive. The proof of Weiss conjecture for $2$-transitive groups is scattered over many papers and hence this result is somewhat
part of folklore, see~\cite[Section~$6$]{PSV1} and the references therein for an overview of the argument. Other examples of strongly graph-restrictive groups are provided by primitive groups with abelian socle, that is, primitive groups of affine type. Recently~\cite{spiga1} it was proved that Weiss conjecture does hold for this class of primitive groups. In many interesting cases, every overgroup of a primitive group of affine type is either affine or $2$-transitive and hence these groups are strongly graph-restrictive. (For instance, most affine groups whose point stabilisers are primitive linear groups satisfy this property.) Now, it would take us too far astray to describe the primitive groups of affine type where each overgroup is either affine or $2$-transitive, thus we simply refer to~\cite{Asch,Asch2} or~\cite{Praeger} for a thorough analysis of the inclusions among primitive groups. Here we simply observe that every primitive group of prime degree is strongly graph-restrictive. In conclusion, if Weiss conjecture proved to be true, then every primitive group is strongly graph-restrictive.

For a strongly graph-restrictive group $L$,
we let $c_*(L)$ denote the maximum of all constants $c(T)$ with $L\le T \le \Sym(\Omega)$.

\begin{theorem}
\label{the:cor}
Let $\Gamma$ be a finite connected $G$-arc-transitive graph of valency at least $3$
such that $G_v^{\Gamma(v)}$ is strongly graph-restrictive.
Then there exists a regular covering projection $\wp \colon \tGa \to \Ga$
with $\tGa$ finite, such that the maximal group that lifts along $\wp$ is $G$ and every automorphism of $\tGa$ projects along $\wp$.
\end{theorem}

\begin{proof}
Let $n$ be the order of $\Gamma$ and let $p$ be a prime with $p >nc_*(G_v^{\Gamma(v)})$.
By Corollary~\ref{cor:3ec}, $\Aut(\Gamma)$ acts faithfully on $\H_1(\Ga;\ZZ_p)$, and then by
Theorem~\ref{the:main},  there exists a regular covering projection $\wp \colon \tGa \to \Ga$
with $\tGa$ finite, such that the maximal group that lifts along $\wp$ is $G$ and that the group of covering transformations of $\wp$ is a $p$-group. 

Let $\tilde{A}$ be the automorphism group of $\tGa$, let $\tilde{G}$ be the lift of $G$ along $\wp$ and let $N=\CT(\wp)$. From Lemma~\ref{lem:maxlift}, $\tilde{G}/N\cong G$ and $\tilde{G}=\N{\tilde{A}}{N}$. Let $\tilde{v}$ be a vertex of $\tGa$, let $v=\wp(\tilde{v})$ and set $c=c_*(G_v^{\Ga(v)})$. Since $G_{v}^{\Gamma(v)}\cong \tilde{G}_{\tilde{v}}^{\tGa(\tilde{v})}$ is strongly graph-restrictive and $\tilde{G}_{\tilde{v}}^{\tGa(\tilde{v})}\le \tilde{A}_{\tilde{v}}^{\tGa(\tilde{v})}$, we have $|\tilde{A}_{\tilde{v}}|\le c$. Since $\tilde{G}$ is transitive on the vertices of $\tGa$, we have  $\tilde{A}=\tilde{A}_{\tilde{v}}\tilde{G}$ and hence $[\tilde{A}:\tilde{G}]=[\tilde{A}_{\tilde{v}}:\tilde{G}_{\tilde{v}}]\le |\tilde{A}_{\tilde{v}}|\le c<p$. Moreover, $[\tilde{G}:N]=|G|=n|G_v|\le nc<p$. Therefore, $[\tilde{A}:N]=[\tilde{A}:\tilde{G}][\tilde{G}:N]$ is not divisible by $p$ and hence $N$ is a Sylow $p$-subgroup of $\tilde{A}$. 

By Sylow theorem, the number of Sylow $p$-subgroups of $\tilde{A}$ is $[\tilde{A}:\N{\tilde{A}}{N}]=[\tilde{A}:\tilde{G}]<p$ and is congruent to $1$ modulo $p$. Therefore, $\tilde{A}=\N{\tilde{A}}N$ and hence $\tilde{A}=\tilde{G}$.
\end{proof}

\begin{corollary}
\label{cor:final}
Let $\Gamma$ be a finite connected $G$-arc-transitive graph of valency at least $3$.  If $G$ acts transitively on the $2$-arcs of $\Gamma$ or if
the valency of $\Gamma$ is prime,
then
there exists a regular covering projection $\wp \colon \tGa \to \Ga$
with $\tGa$ finite, such that the maximal group that lifts along $\wp$ is $G$ and every automorphism of $\tGa$ projects along $\wp$. 
\end{corollary}

We conclude daring to conjecture that the requirement of $G_v^{\Gamma(v)}$ being strongly graph-restrictive in Theorem~\ref{the:cor} is not  necessary.

\begin{conjecture}
\label{conj}
Let $\Gamma$ be a finite connected graph such that the induced action of $\Aut(\Gamma)$
 on $\H_1(\Ga;\ZZ)$ is faithful, and  let $G\le \Aut(\Gamma)$.  
 Then there exists a regular covering projection $\wp \colon \tGa \to \Ga$
with $\tGa$ finite, such that the maximal group that lifts along $\wp$ is $G$ and such that $\Aut(\tGa)$ projects along $\wp$.
\end{conjecture}

\smallskip

{\bf Acknowledgements.} The  first author gratefully acknowledges financial support of the Slovenian Research Agency, ARRS, research program no.\ P1-0294.

\end{document}